\documentclass[11pt]{article}
\usepackage[english]{babel}
\usepackage{latexsym,amsfonts,amsmath,enumerate,graphics,enumerate,amsthm,tikz,hyperref}
\usepackage[affil-it]{authblk}
\usepackage{enumerate,amsthm,dsfont,pstricks}
\usepackage{latexsym,amsfonts,amsmath,amssymb,amscd,wrapfig,graphicx,curves}

\newtheorem{theorem}{Theorem}[section]
\newtheorem{lemma}[theorem]{Lemma}
\newtheorem{proposition}[theorem]{Proposition}
\newtheorem{corollary}[theorem]{Corollary}

\theoremstyle{definition}
\newtheorem{definition}[theorem]{Definition}

\theoremstyle{remark}
\newtheorem{remark}[theorem]{Remark}

\let\phi=\varphi
\let\theta=\vartheta
\def\epsilon{\varepsilon}
\def\d{\,\mathrm{d}}
\def\Ob{{\overline{\Omega}}}
\def\N{\mathbb{N}}
\def\R{\mathbb{R}}

\def\eps{\varepsilon}
\def\Ob{{\overline{\Omega}}}
\def\0{\mathbf{0}}

\newcommand{\comment}[1]{}
\newcommand{\norm}[1]{\left\Vert #1 \right\Vert} 

\numberwithin{equation}{section}

\let\epsilon=\varepsilon
\let\phi=\varphi
\let\theta=\vartheta

\makeatletter
\def\@maketitle{%
  \newpage
  \null
  \vskip 2em%
  \begin{center}%
  \let \footnote \thanks
    {\Large\bfseries \@title \par}%
    \vskip 1.5em%
    {\normalsize
      \lineskip .5em%
      \begin{tabular}[t]{c}%
        \@author
      \end{tabular}\par}%
    \vskip 1em%
    {\normalsize \@date}%
  \end{center}%
  \par
  \vskip 1.5em}
\makeatother

\begin{document}

\title{\sc \LARGE Isometries of infinite dimensional Hilbert geometries}

\author{Bas Lemmens%
\thanks{Email: \texttt{B.Lemmens@kent.ac.uk}; Supported by EPSRC grant EP/J008508/1 (Corresponding author)}}
\affil{School of Mathematics, Statistics \& Actuarial Science,
University  of Kent, Canterbury, Kent CT2 7NX, UK}

\author{Mark Roelands%
\thanks{Email: \texttt{mark.roelands@gmail.com}  Supported by EPSRC grant EP/J500446/1}}
\affil{Unit for BMI, North-West University, Private Bag X6001-209, Potchefstroom 2520, South Africa}

\author{Marten Wortel%
\thanks{Email: \texttt{marten.wortel@gmail.com}; Supported by EPSRC grant EP/J008508/1}}
\affil{Unit for BMI, North-West University, Private Bag X6001-209, Potchefstroom 2520, South Africa}
\maketitle
\date{}

\begin{abstract}
In this paper we extend two classical  results  concerning the isometries of strictly convex Hilbert geometries, and the characterisation of the isometry groups of Hilbert geometries on finite dimensional simplices, to infinite dimensions. The proofs rely on a  mix of geometric and functional analytic methods.  
\end{abstract}



\section{Introduction}

In \cite{Hil} Hilbert introduced a collection of metric spaces that are natural deformations of finite dimensional real hyperbolic spaces. Although Hilbert
limited his construction to finite dimensions, it has a straightforward extension to infinite dimensional spaces. Indeed, let $\Omega$ be a convex subset of a  (not necessarily
finite-dimensional) real vector space $Y$, and suppose that  for each $x \not= y \in \Omega$ the straight line $\ell_{xy}$ through $x$ and
$y$ has the property that $\Omega \cap \ell_{xy}$ is an open and bounded line segment in $\ell_{xy}$. In that case one can define Hilbert's metric on  $\Omega$  as
follows. For $x \neq y$ in $\Omega$, let $x'$ and $y'$ be the end-points of the segment
$\ell_{xy} \cap \Omega$ such that $x$ is between $x'$ and $y$, and $y$ is between $y'$ and $x$. Now {\em Hilbert's metric on $\Omega$} is given by 
\[
\delta_H(x,y) :=\log [x',x,y,y']\mbox{\quad for $x\neq y$ in $\Omega$,} 
\]
where 
\[
 [x',x,y,y'] := \frac{|x'-y|}{|x'-x|}\frac{|y'-x|}{|y'-y|}
\]
is the {\em cross-ratio}, and $\delta_H(x,x)=0$ for all $x\in \Omega$. 
The metric space $(\Omega,\delta_H)$ is usually called the {\em Hilbert geometry on $\Omega$}. In particular, the open unit ball in an infinite dimensional Hilbert space equipped with $\frac{1}{2}\delta_H$, is precisely Klein's model of the infinite dimensional hyperbolic space. A recent extensive overview of the theory of Hilbert geometries can be found in \cite{Handb}. 

The isometries between finite dimensional Hilbert geometries are well understood. They have been  studied intensively in the past decade by Bosch\'{e} \cite{Bo}, de la Harpe \cite{dlH}, Lemmens and Walsh \cite{LW}, Matveev and Troyanov \cite{MT}, Speer \cite{Sp}, and Walsh \cite{W2}. The purpose of this paper is to  analyse the isometries of infinite dimensional geometries 
on strictly convex domains and infinite dimensional simplices. To date there are only a few works on  infinite dimensional Hilbert geometries. We should mention the work  \cite{Mol1} by  Moln\'ar in which the group of Hilbert's metric isometries on the projective domain of the cone of positive self-adjoint operators on a complex Hilbert space is determined, and  \cite{BIM, MP} in which the  isometries of infinite dimensional hyperbolic space are studied. 

It is well known that Hilbert's metric has important applications in the analysis of linear, and nonlinear, operators on cones both in finite and infinite
dimensions, see \cite{LNBook,LNSurv,NMem}. In mathematical analysis one often works with Birkhoff's version of Hilbert's metric, which provides a slightly more
general set up than the one outlined above. Birkhoff's version of Hilbert's metric, denoted $d_H$,  is a metric on the set of rays in the interior, $C^\circ$, of  a closed cone $C$ in a normed space $X$. If  there exists a linear functional $\phi$ on $X$ with $\phi(x)>0$ for all $x\in C\setminus\{0\}$, then $\delta_H$ and $d_H$ coincide on $\Sigma_\phi=\{x\in C^\circ\colon \phi(x)>0\}$, see for example \cite[Theorem 2.1.2]{LNBook}. In general, however,  there may not exists such a linear functional, see Remark 2.3 for more details, and in this respect Birkhoff's version is more general. Another advantage of using cones is that Hilbert's metric can be expressed in terms of the partial order induced by the cone, and one can use ideas from the theory  of partially ordered vector spaces. In this paper we will be mainly working with Birkhoff's version of Hilbert's metric.

The paper has the following outline. In  Section~\ref{sec:general_birkhoff_theory} we shall, beside introducing the relevant definitions,  explain the relation between Hilbert's metric and Birkhoff's version of Hilbert's metric. Among other things we shall construct for a given Hilbert metric space $(\Omega,\delta_H)$ in $Y$, a real normed vector space $X$ containing $Y$ such that on $\Omega$ the relative norm topology of $X$ coincides with the Hilbert's metric topology. 

Subsequently we prove in Section 3 the following theorem, which generalises \cite[p. 163 (29.1)]{BK} and  \cite[Proposition 3]{dlH}. 
\begin{theorem}\label{thm:infinite_strictly_convex} 
If $\Omega_1$ and $\Omega_2$ are strictly convex Hilbert geometries, and  $f\colon\Omega_1\to\Omega_2$ is an isometry of $(\Omega_1,\delta_H)$ into
$(\Omega_2,\delta_H)$, then $f$ is a projective linear homomorphism. 
\end{theorem}
The notion of a projective linear homomorphism will be given in Definition \ref{prohom}. 

As we are working in infinite dimensions, one cannot use a projective basis and the fundamental theorem of projective geometry to prove Theorem \ref{thm:infinite_strictly_convex}. Instead we establish an extension result, Proposition \ref{extension of collineation}, which can be combined with Zorn's Lemma to prove Theorem \ref{thm:infinite_strictly_convex}.

The second main result is proved in Section 4, and concerns the isometries of Hilbert geometries on infinite dimensional simplices. A natural generalisation of finite dimensional simplices to infinite dimensions is the set 
\[ \Delta(K, \mu)^\circ := \{f\in C(K)\colon f(x)>0 \mbox{ for all }x\in K\mbox{ and $\int f \d \mu=1$}\}, \]
where $C(K)$ is the set of continuous functions on a compact Hausdorff space $K$ and $\mu$ is a finite, strictly positive, Borel measure on $K$, i.e., $\mu(K)<\infty$ and $\int f\mathrm{d}\mu>0$ for all $f\in C(K)$ with $f\neq 0$ and $f(x)\geq 0$ for all $x\in K$. Note that the set 
\[
 \Delta^\circ_\infty:= \{ x \in \ell^\infty \colon x_i>0\mbox{ for all $i \in \N$ and $\phi(x) := \sum_{i=1}^\infty 2^{-i} x_i =1$}\} 
\] is a special case. Indeed, as $\ell^\infty$ is an abstract $M$-space with an order unit, it follows from Kakutani's representation theorem that $\ell^\infty$  is isometrically order-isomorphic to $C(K)$ for some compact Hausdorff space $K$, see \cite[Theorem 1.b.6]{LT}.  

For the Hilbert geometries $(\Delta(K,\mu)^\circ, \delta_H)$ we have the following result. 
\begin{theorem}\label{thm:1.2}
If $K_1$ and $K_2$ are compact Hausdorff spaces with finite, strictly positive, Borel measures $\mu_1$ and $\mu_2$, respectively, then $h\colon
\Delta(K_1,\mu_1)^\circ\to
\Delta(K_2,\mu_2)^\circ$ is a surjective Hilbert's metric  isometry if and only if there exist $\epsilon \in\{-1,1\}$, a homeomorphism $\theta\colon K_2\to
K_1$, and $g\in C(K_2)$ with $g(x)>0$ for all $x\in K_2$ such that 
\[
h(f)=\frac{g\cdot (f\circ \theta)^\epsilon}{\int g\cdot (f\circ \theta)^\epsilon \d \mu_2}.
\]  
If $(K_1, \mu_1) = (K_2, \mu_2) = (K,\mu)$ and $|K|\geq 3$, then the isometry group is given by 
\[
\mathrm{Isom}(\Delta(K,\mu)^\circ, \delta_H) \cong  \overline{C(K)} \rtimes (C_2 \times \mathrm{Homeo}(K)),
\]
where $C_2$ is the cyclic group of order $2$, $\overline{C(K)}:= C(K)/\mathbb{R}\mathbf{1}$, and $\mathbf{1}$ is the constant one function on $K$. 
\end{theorem}
Theorem \ref{thm:1.2} generalises the characteristion of the isometries of the Hilbert geometry on finite dimensional simplices obtained in \cite{dlH}, see also
\cite [Theorem 5.1]{FM} and \cite[Theorem 1.2]{LW}.  As a  direct consequence we obtain the following result. 
\begin{corollary}
If $K_1, K_2$ are compact Hausdorff spaces,  then $(\Delta(K_1,\mu_1)^\circ,d_H)$ and $(\Delta(K_2,\mu_2)^\circ,d_H)$ are isometric if and only if $K_1$ and $K_2$ are homeomorphic.
\end{corollary}

\section{Preliminaries}\label{sec:general_birkhoff_theory}
In this section we introduce the basic concepts and preliminary results. A cone $C$ in a vector space $X$ is a convex set such that $\lambda C\subseteq C$ for all $\lambda\geq 0$ and $C\cap (-C)=\{0\}$. A cone $C$ induces a partial
ordering $\leq_C$ on $X$ by $x\leq_C y$ if $y-x\in C$. The cone is said to be {\em Archimedean} if for each $x\in X$ and $y\in C$ with $nx\leq_C y$ for all
$n=1,2,3,\ldots$ we have  that $x\leq_C 0$. 
An element $u\in C$ is called an {\em order unit} if for each $x\in X$ there exists $\lambda>0$ such that $x\leq_C\lambda u$. The triple $(X,C,u)$ is called an {\em order unit space}, if $C$ is an Archimedean cone in $X$ and $u$ is an  order unit for $C$. 

Given an order unit space $(X,C,u)$, the space $X$ can be equipped with the so-called {\em order unit norm}, 
\[
\|x\|_u:=\inf\{\lambda>0\colon -\lambda u\leq_C x\leq_C\lambda u\}.
\]
By \cite[Theorem 2.55(2)]{AT}, $C$ is closed under $\|\cdot\|_u$. Note also that $\|\cdot\|_u$ is a {\em monotone norm}, i.e., $\|x\|_u\leq\|y\|_u$ for all $0\leq_C x\leq_C y$, and hence 
$C$ is normal with respect to $\|\cdot\|_u$. Recall that a cone $C$ in a normed space $(X,\|\cdot\|)$ is called {\em normal} if there exists a constant $\kappa>0$ such that $\|x\|\leq \kappa \|y\|$ whenever $0\leq_C x\leq_C y$. Furthermore, $C$ has nonempty interior, $C^\circ$, with respect to $\|\cdot\|_u$, as the following lemma shows. 
\begin{lemma}\label{lem:order_units_interior}
 If $(X,C,u)$ is an order unit space, then the set of order units of $C$ coincides with $C^\circ$.
\end{lemma}
\begin{proof}
 Every interior point is an order unit by \cite[Lemma~2.5]{AT}. Conversely, if $x \in C$ is an order unit,
then there exists $M > 0$ such that $u / M \leq x$. If $\| y \|_u \leq 1 / M$, then $x-y \geq x-u/ M \geq 0$, so $x \in C^\circ$.
\end{proof}

Throughout the paper we shall always assume that an order unit space is equipped with the order unit norm. 

A linear functional $\phi\colon X\to \mathbb{R}$ on an order unit space $(X,C,u)$ is said to be  {\em positive} if $\phi(C)\subseteq [0,\infty)$. It is said to {\em strictly positive} if  $\phi(C\setminus\{0\})\subseteq (0,\infty)$.  A positive linear functional $\phi$ with $\phi(u)=1$ is called a {\em state} of $(X,C,u)$. Strictly positive states are always continuous and $\|\phi\|=1$, as can be seen from the following lemma applied to $X_2=\mathbb{R}$. 
\begin{lemma}\label{lem:automatic_continuity}
 Let $(X_1, C_1, u_1)$ and $(X_2,C_2 ,u_2 )$ be order unit spaces.  If $T \colon X_1 \to X_2$ is a linear map such that $T(C_1)\subseteq C_2$, then $T$ is
continuous with $\norm{T} = \norm{Tu_1}_{u_2}$.
\end{lemma}
\begin{proof}
 If $\norm{x}_{u_1} \leq 1$, then $-u_1 \leq_{C_1} x \leq_{C_1} u_1$, so that $-Tu_1 \leq_{C_2} Tx \leq_{C_2} Tu_1$. The statement now follows from the definition of $\norm{\cdot}_{u_2}$.
\end{proof}

\begin{remark}
 As mentioned in the introduction, there may not exist a strictly positive functional for a given order unit space $(X,C,u)$. Indeed, if $X$ is the vector space of bounded functions on an uncountable set $K$, $C$ is the  Archimedean cone of functions taking nonnegative values everywhere, and $u$ is the constant one function on $K$, then no strictly positive functional exists, see \cite[Exercise~6,~Section~1.7]{AT}. However, if $X$ is separable, then a strictly positive state always exists. Indeed, in that case the unit
ball $B_{X^*}$ of $X^*$ (which is weak*-compact) is weak*-metrizable, and so there exists a sequence of states $\phi_n$ which is weak*-dense in the set of all
states. A standard argument then shows that the state $\phi := \sum_{n=1}^\infty 2^{-n} \phi_n$ is strictly positive.
\end{remark}

If $(X,C,u)$ is an order unit space and $x,y\in C^\circ$, then it follows from  Lemma~\ref{lem:order_units_interior} that there exist
$0<\alpha\leq \beta$ such that $\alpha y\leq_C x\leq_C \beta y$, and hence we can define 
\[
M(x/y):=\inf\{\beta >0\colon x\leq_C \beta y\}<\infty.
\]
On $C^\circ$ {\em (Birkhoff's version of) Hilbert's metric} is given by 
\[
d_H(x,y) :=\log\left(M(x/y)M(y/x)\right).
\] 
It is easy to verify that $d_H(\lambda x,\mu y)=d_H(x,y)$ for all $\lambda,\mu>0$, and hence $d_H$ is not  a metric. However, it can be shown that $d_H$ is a
metric between pairs of rays in $C^\circ$, see \cite[Lemma 2.1]{LNSurv}. We shall denote the  projective spaces obtained by identifying points on rays inside
$C^\circ$ by $P(C^\circ)$. So,
$(P(C^\circ),d_H)$ is a metric space. 

In case there exists a strictly positive state $\varphi\colon X\to\mathbb{R}$, we can
identify $P(C^\circ)$ with the cross section
\[
\Sigma_\varphi:=\{x\in C^\circ\colon \varphi(x) =1\}.
\] 
We shall use the following notation. For $x\in C^\circ$, we write $[x]:= x/\varphi(x)\in \Sigma_\varphi$. The boundary of $\Sigma_\varphi$ relative to the
affine space $\{x\in X\colon \varphi(x)=1\}$ in $(X,\|\cdot\|_u)$  is denoted by $\partial \Sigma_\varphi$. Its closure will be denoted by $\overline{\Sigma}_\phi$. 

The following theorem shows the relation between Hilbert geometries and order unit spaces with strictly positive functionals.
\begin{theorem}\label{thm:cone_classification}
If $Y$ is a vector space and $\Omega \subset Y$ is a convex set on which $\delta_H$ is well defined, then there exists an order unit space $(X,C,u)$ and
a strictly positive state $\phi$ on $X$ such that $\Omega$ is affine isomorphic to $\Sigma_\phi$.

Conversely, if $(X,C,u)$ is an order unit space with a strictly positive state $\phi$, then $\Sigma_\phi$ is a convex set on which $\delta_H$ is well defined.
\end{theorem}
\begin{proof}
Let $\Omega \subset Y$ be a convex set on which $\delta_H$ is well defined, i.e., $\Omega \cap \ell_{xy}$ is open and bounded for every $x,y \in
\Omega$ with $x \not= y$. By translating we may assume without loss of generality that $0 \in \Omega$ and, by restricting
to the span of $\Omega$, we may also assume that $Y =\mathrm{Span}\,\Omega$. 

We claim that these assumptions imply that $\Omega \cap \ell_{0y}$ is open and bounded for
every nonzero $y \in Y$. Indeed, let $y \in Y$ be nonzero, then $y = \sum_{i=1}^n \lambda_i y_i$ for some $y_i \in \Omega$ and $\lambda_i \in \R$, since
$\Omega$ spans $Y$. In this representation we may assume that each $\lambda_i > 0$. Indeed, if $\lambda_i < 0$ for some $i$, then we can replace $y_i$ by $-\eps
y_i$ for some small $\epsilon$ (since $\ell_{0 y_i} \cap \Omega$ is open) and $\lambda_i$ by $-\eps^{-1} \lambda_i > 0$. Now if $\lambda := \sum_{i=1}^n
\lambda_i$, then each $\lambda^{-1} \lambda_i>0$ and they sum to $1$. So, 
\[ \lambda^{-1} y = \sum_{i=1}^n \lambda^{-1} \lambda_i y_i \in \Omega \]
by convexity of $\Omega$, and hence $\Omega \cap \ell_{0y} = \Omega \cap \ell_{0 (\lambda^{-1} y)}$ is open and bounded.

By considering lines through $0$, this implies that $\Omega$ is absorbing, i.e., for
every $y \in Y$, there exists an $\eps > 0$ such that $\lambda y \in \Omega$ for all $|\lambda| < \eps$. Define
\[ \Ob := \{y \in Y \colon \lambda y \in \Omega \mbox{ for all } 0 \leq \lambda < 1 \} \]
and note that  $\Ob$ is convex.

Let $X:=Y \oplus \R$ and consider the cone 
\[ C_\Omega := \{ \lambda (y,1) \in X \colon \lambda \geq 0\mbox{ and } y \in \Ob \}. \]
We now show that  $C_\Omega$ is Archimedean and $(0,1)\in X$ is an order unit. 

Let  $(y, \lambda) \in X$ with $y\in Y$ and $\lambda\in\mathbb{R}$. To show  that $(0,1)\in X$ is an order unit, we have to find $M > 0$ with $-M(0,1) \leq_{C_\Omega} (y, \lambda) \leq_{C_\Omega}
M(0,1)$. These inequalities hold if and only if $(y, \lambda + M), (-y, M - \lambda) \in C_\Omega$, which is equivalent to $y / (M+\lambda), -y / (\lambda - M)
\in
\Ob$. But the existence of such an $M$ now follows from the fact that $\Ob$ is absorbing.

To show that $C_\Omega$ is Archimedean, it suffices to prove that $n(y, \lambda) \leq_{C_\Omega} (0,1)$ for all large enough $n$ implies that $-(y, \lambda) \in C_\Omega$,
as $(0,1)$ is an order unit. The assumed inequality is equivalent with $(-ny, 1-n\lambda) \in C_\Omega$, which shows that $\lambda \leq 0$. If $\lambda = 0$,
then $-ny \in \Ob$, which implies by the boundedness of $\ell_{0y} \cap \Omega$ that $y =
0$ and so $-(y,\lambda) = (0,0) \in C_\Omega$. If $\lambda < 0$, then by scaling we may assume that $\lambda = -1$, and then dividing the inequality by $1+n$
yields
\[ \frac{n}{1+n} (-y) \in \Ob. \]
Hence $-y \in \Ob$, and so $-(y, \lambda) = (-y, 1) \in C_\Omega$.

Obviously, the linear functional $\varphi \colon X \to \R$ defined by $\phi((y,s)):=s$ is strictly
positive with respect to $C_\Omega$ and $\phi((0,1)) = 1$, so $\phi$ is a strictly positive  state.

Let us now show that $\Omega$ is affine isomorphic to  $\Sigma_\varphi := \{x \in C^\circ_\Omega \colon \phi(x) = 1\}$. First assume that $y\in \overline{\Omega}\setminus\Omega$. Note that if we can prove that $(y,1)$ is not an order unit, then $(y,1)\not\in C^\circ_\Omega$ by Lemma \ref{lem:order_units_interior}. 
Suppose by way of contradiction that $(y,1)$ is an order unit. Then there exists $M>1$ such that 
$(0,1)\leq_{C_\Omega} M(y,1)$, which is equivalent with $\frac{M}{M-1}y\in\overline{\Omega}$. Taking $0<\lambda :=\frac{M-1}{M}<1$ we get that $y = \lambda \frac{M}{M-1} y\in\Omega$, which is absurd. 
Now suppose that $y\in\Omega$. We need to show that $(y,1)\in C^\circ_\Omega$. As $(0,1)\in C^\circ_\Omega$ is an order unit, it suffices to show by Lemma \ref{lem:order_units_interior} that there exists $M>1$ such that $(0,1)\leq_{C_\Omega} M(y,1)$. Recall that $\ell_{0y}\cap\Omega$ is an open subset of $\ell_{0y}$. Hence there exists $M>1$ such that $\frac{M}{M-1}y\in \ell_{0y}\cap\Omega$, which implies that $(0,1)\leq_{C_\Omega} M(y,1)$. We conclude that  $\Omega$ is affine isomorphic to  $\Sigma_\varphi := \{x \in C^\circ_\Omega \colon \phi(x) = 1\}$. 

To prove the second part, we note that for distinct $x,y\in\Sigma_\varphi$  the points  
$w_x:=x-M(y/x)^{-1}y$ and $w_y:= y - M(x/y)^{-1}x$ are in $\partial C\setminus\{0\}$, as $C$ is closed. So, $[w_x]$ and $[w_y]$ are the end-points of the
straight line segment $\ell_{xy}\cap \Sigma_\varphi$. 
\end{proof}

\begin{remark}
In the proof of Theorem \ref{thm:cone_classification}, the vector space $Y$ is a subspace of $X$, and so it inherits the norm $\norm{\cdot}_u$ from $X$. Thus,
for $y \in Y$,
\[ \norm{y} = \inf \{ \lambda > 0 \colon - \lambda(0,1) \leq (y,0) \leq \lambda(0,1) \}. \]
The condition on $\lambda$ is equivalent with $(y,\lambda), (-y, \lambda) \in C_\Omega$, which in turn is equivalent with $y/ \lambda, -y/ \lambda \in \Ob$,
and so
\[ \norm{y} = \inf \{ \lambda > 0 \colon y \in \lambda(\Ob \cap - \Ob) \}. \]
Hence this norm equals the Minkowski functional of $\Ob \cap -\Ob$.
\end{remark}

The advantage of working  with cones is that we can use $d_H$ instead of $\delta_H$ and apply ideas from the theory of partially ordered vector spaces. Indeed, the following result, which  goes back to Birkhoff \cite{Bi}, is well known, see for example \cite[Theorem 2.12 and Corollary 2.5.6]{LNBook}. 
\begin{lemma} \label{lem:normal}
If $(X,C,u)$ is an order unit space with strictly positive state $\phi$, then on  $\Sigma_\varphi$ the metrics $d_H$ and $\delta_H$ coincide. Moreover, the Hilbert's metric topology on $\Sigma_\varphi$ coincides with
the order unit norm topology on $X$. 
\end{lemma}

If $(X_1,C_1,u_1)$ and $(X_2,C_2,u_2)$ are two order unit spaces and $T\colon
X_1\to X_2$ is  a linear map,  we say that $T$ is {\em bi-positive} if $Tx\in C_2$ if and only if $x\in C_1$. Note that a bi-positive linear map $T\colon X_1\to X_2$  is always  injective, as $Tx = 0$ implies that $x$ and $-x$ in $C_1$, so that $x=0$. Also note that if there exists $x\in C_1^\circ$ such that $Tx\in C^\circ_2$, then $T(C_1^\circ)\subseteq C_2^\circ$. Indeed, if $y\in C^\circ_1$, then $y$ is an order unit by Lemma \ref{lem:order_units_interior}, and hence there exists $\lambda>0$ such that $x\leq_{C_1}\lambda y$. It follows that $Tx\leq_{C_2}\lambda Ty$, so that $Ty$ is an order unit, as $Tx$ is an order unit, and hence $Ty\in C^\circ_2$. 

It is easy to check that a bi-positive linear map $T\colon X_1\to X_2$ induces a Hilbert's metric isometry, as $M(Tx/Ty)=M(x/y)$ for all $x,y\in C_1^\circ$. In that case we  shall denote the induced map between the projective spaces $P(C^\circ_1)$ and $P(C^\circ_2)$ by $[T]$.  

\begin{definition}\label{prohom}
An isometry $f$  from  $(P(C_1^\circ),d_H)$ into $(P(C_2^\circ),d_H)$ is called a {\em projective linear homomorphism} if there exists a bi-positive linear map $T\colon X_1\to X_2$ such that $f=[T]$ on $P(C^\circ_1)$. 
\end{definition}

Before we start the proof of Theorem \ref{thm:infinite_strictly_convex}, we collect some final pieces of notation. Recall that the image of  a map $\gamma$ from a (possibly unbounded) interval $I\subseteq\mathbb{R}$ into $(\Sigma_\varphi, d_H)$ is a {\em geodesic} if 
\[
d_H(\gamma(t),\gamma(s)) = |t-s|\mbox{ for all }t,s\in I.
\]

Given a straight line $\ell_{xy}$ through $x\neq y$ in $\Sigma_\phi$, we write 
 $\ell^+_{xy}:=\ell_{xy}\cap\Sigma_{\varphi}$. Also for  $x,y\in X$ the closed and open line segments are, respectively,  denoted by $$[x,y]:=\{tx+(1-t)y:0\le
t\le 1\}\quad\mbox{and}\quad(x,y):=\{tx+(1-t)y:0<t<1\}.$$ The half-open intervals are 
defined in a similar way. For each $x,y\in \Sigma_\varphi$ the segment $[x,y]$ is a geodesic in $(\Sigma_\varphi,d_H)$, see \cite{Hil}. It is, however, in general not the only geodesic. 
The unique geodesics in $(\Sigma_\varphi,d_H)$ are characterised as follows, see \cite{Hil} or \cite[Proposition 2]{dlH}. 
\begin{lemma}\label{lem:uniq} Let $(X,C,u)$ be an order unit space with strictly positive state $\phi$.  If
$x,y\in \Sigma_\varphi$ and $x',y'\in\partial \Sigma_\varphi$ are the end points of 
$\ell_{xy}\cap \Sigma_\varphi$, then $[x,y]$ is the unique geodesic connecting $x$ and $y$ in
$(\Sigma_\varphi,d_H)$ if and only if there exist no open line segments $I_{x'}$ through $x'$  and $I_{y'}$ through $y'$ in $\partial\Sigma_\varphi$ such that
the affine span of $I_{x'}\cup I_{y'}$ is 2-dimensional. 
\end{lemma}

\section{Strictly convex Hilbert geometries}
We prove Theorem \ref{thm:infinite_strictly_convex}. Throughout this section we shall
assume that $(X_1, C_1, u_1)$ and $(X_2, C_2, u_2)$ are order unit spaces,  with strictly positive states $\varphi_i \colon
X_i\to\mathbb{R}$ for $i=1,2$. For simplicity  we write 
\[
\Sigma_i :=\{x\in C_i^\circ\colon \varphi_i(x) =1\}  \quad \mbox{and} \quad \norm{\cdot}_i := \norm{\cdot}_{u_i} \mbox{\quad for $i=1,2$. }
\]

Recall that $\Sigma_i$, $i=1,2$, is {\em strictly convex} if for each $x,y\in\partial\Sigma_i$ we have that $(x,y)\subseteq \Sigma_i$. In that case, it follows from
Lemma \ref{lem:uniq} that the metric spaces $(\Sigma_i,d_H)$, $i=1,2$, are uniquely geodesic. 
Thus, any isometry $f$ of $(\Sigma_1,d_H)$ into $(\Sigma_2,d_H)$ has to map line segments to line segments if $(\Sigma_1,d_H)$  and $(\Sigma_2,d_H)$  are
strictly convex. We shall show that any isometry that maps line segments to line segments must be a projective linear homomorphism, which implies Theorem
\ref{thm:infinite_strictly_convex}. We begin with the following lemma, which generalises \cite[p.101]{dlH}.

\begin{lemma}\label{isometry in strictly convex cone} If $f\colon(\Sigma_{1},d_H)\to(\Sigma_2,d_H)$ is an isometry that maps line segments to line segments, then
$f$ has a unique continuous extension to a map from $(\overline{\Sigma}_{1},\|\cdot\|_1)$ into $(\overline{\Sigma}_{2},\|\cdot\|_2)$. Furthermore, this
extension is injective.
\end{lemma}
\begin{proof} 
The uniqueness of the continuous extension follows from the density of $\Sigma_1$ in $\overline{\Sigma}_1$. Convergence in this proof will always be in the order unit norm; recall
that on $\Sigma_i$, the order unit norm topology  and $d_H$-topology coincide by Lemma~\ref{lem:normal}.

Fix an element $p\in\Sigma_1$ and let $x\in\partial\Sigma_1$. The line segment $[p,x)$ is mapped onto the line
segment $[f(p ),\xi)$, for some $\xi\in\partial\Sigma_2$, because  $$d_H(p,(1-t)p+tx)\to\infty$$ as $t\uparrow 1$. We define $f(x):=\xi$.  

To show continuity of the extension, let $(x_n)_n$ be a sequence in $\overline{\Sigma}_{1}$ converging to $x$. Then $y_n := (p+x_n)/2 \to (p+x)/2 =: y$, and
so $f(y_n) \to f(y)$. Let $s > 1$ be such that $sf(y) + (1-s)f(p) \notin \overline{\Sigma}_2$. Then for some $N \geq 1$ and all $n \geq N$,
\begin{equation}\label{eq:outside_sigma}
 sf(y_n) + (1-s) f(p) \notin \overline{\Sigma}_2.
\end{equation}
Since $f$ maps line segments to line segments, there exist $s_n$ such that 
\[ s_n f(y_n) + (1-s_n) f(p) = f(x_n) \in \overline{\Sigma}_2, \]
and combining this with \eqref{eq:outside_sigma} yields $s_n < s$ for all $n \geq N$.

Now suppose that $f(x_n)$ does not converge to $f(x)$. By passing to a subsequence three times, we find a subsequence $(x_{n_k})_k$ such that $f(x_{n_k})$
stays away from $f(x)$,  $s_{n_k} \to r \in [0,s]$, and that either all $x_{n_k} \in \partial \Sigma_1$ or all $x_{n_k} \in \Sigma_1$. It follows that
\[ f(x_{n_k}) = s_{n_k} f(y_{n_k}) + (1-s_{n_k}) f(p) \to r f(y) + (1-r) f(p) \in \overline{\Sigma}_2. \]
We claim that $rf(y)+(1-r)f(p) \in \partial\Sigma_2$. Indeed, if all $x_{n_k} \in \partial \Sigma_1$, then $f(x_{n_k}) \in \partial \Sigma_2$ and the claim
follows from the closedness of $\partial \Sigma_2$. If all $x_{n_k} \in \Sigma_1$, then the fact $f$ is a $d_H$-isometry combined with $d_H(x_{n_k},p) \to
\infty$ yields $d_H(f(x_{n_k}), f(p)) \to \infty$, and so $rf(y) + (1-r)f(p) \in \partial \Sigma_2$. 

Hence $f(x) = rf(y) + (1-r)f(p)$ by construction of $f(x)$, and so $f(x_{n_k}) \to f(x)$, which is impossible since $f(x_{n_k})$ stays away
from $f(x)$. Therefore the extension of $f$ is continuous.
 
To show  injectivity suppose that $x,y\in\overline{\Sigma}_1$ are such that $x\neq y$. Note that 
$[f(p),f(x)]\neq[f(p),f(y)]$, since $f$ is injective on $\Sigma_1$; so,  $f(x)\neq f(y)$.
\end{proof}

The following lemma is essentially part (iii) of the lemma on page 101 in \cite{dlH}, and will be useful in the sequel.

\begin{lemma}\label{two-dimensional base change}
Let $x,y\in \Sigma_1$ and let $x',y'\in\partial \Sigma_1$ be the end points of $\ell_{xy}^+$ such that $x$ is between $x'$ and $y$, and $y$ is between $y'$ and
$x$. 
Suppose that $f\colon(\Sigma_1,d_H)\to(\Sigma_2,d_H)$ is an isometry that maps line segments to line segments. If $f(x)'$ and $f(y)'$ are the end points of
$\ell_{f(x)f(y)}^+$ such that $f(x)$ is between $f(x)'$ and $f(y)$, and $f(y)$ is between $f(y)'$ and $f(x)$, then  there exists a 
linear map $S\colon\mathrm{Span}\{x',y'\}\to\mathrm{Span}\{f(x)',f(y)'\}$ satisfying $[Sx']=f(x)'$, $[Sy']=f(y)'$, $[Sz]=f(z)$ for all $z\in (x',y')$, and $S$ is bi-positive with respect to  $C_1\cap\mathrm{Span}\{x',y'\}$ and $C_2\cap\mathrm{Span}\{f(x)',f(y)'\}$.
\end{lemma}
\begin{proof} Let $0<s,t<1$ be such that $x=tx'+(1-t)y'$ and $f(x)=sf(x')+(1-s)f(y')$. Since $\{tx',(1-t)y'\}$ and $\{sf(x'),(1-s)f(y')\}$ define bases for
$\mathrm{Span}\{x',y'\}$ and $\mathrm{Span}\{f(x)',f(y)'\}$, and the cones $C_1\cap\mathrm{Span}\{x',y'\}$ and $C_2\cap\mathrm{Span}\{f(x)',f(y)'\}$ are the
positive span of these basis elements, we have a bijective and bi-positive linear map $S\colon \mathrm{Span}\{x',y'\}\to\mathrm{Span}\{f(x)',f(y)'\}$ defined
by 
\[
S(\alpha tx'+\beta (1-t)y'):=\alpha sf(x)'+\beta (1-s)f(y)'.
\] 
Note that $[Sx']=f(x)'$, $[Sy']=f(y)'$ and $Sx = f(x)$. 

Let $z\in [x,x')$ and note that, as $f$ maps line segments to line segments and  $f$ is an isometry, we have that
\[
[f(x)',f(z),f(x),f(y)']=\exp(d_H(f(z),f(x)))=\exp(d_H(z,x))=[x',z,x,y'].
\]

As the cross ratio is a projective invariant, we know that  
\[
\left[[Sx'],[Sz],[Sx],[Sy']\right]= [Sx',Sx,Sy,Sy']=[x',z,x,y'], 
\] 
which combined with the previous equality gives 
 \[
\left[[Sx'],[Sz],[Sx],[Sy']\right]=[f(x)',f(z),f(x),f(y)']= \left[[Sx'],f(z),[Sx],[Sy']\right].
\] 
This implies that $[Sz]=f(z)$. Interchanging the roles of $x'$ and $y'$ finally gives $[Sz]=f(z)$ for all  $z\in (x',y')$.
\end{proof}

\begin{proposition}\label{extension of collineation}Let $f\colon(\Sigma_1,d_H)\to(\Sigma_2,d_H)$ be an isometry that maps line segments to line segments and
$Y\subseteq X_1$ be a subspace such that $Y\cap\Sigma_1\neq \emptyset$. If $T\colon Y\to X_2$ is a bi-positive  linear map
such that $[Ty]=f(y)$ for all $y\in Y\cap\Sigma_1$, then for $z\in\Sigma_1\setminus Y$ there is a bi-positive linear extension $$\hat{T}\colon
Y\oplus\mathrm{Span}\{z\}\to X_2$$ of $T$ such that $[\hat{T}y]=f(y)$ for all $y\in(Y\oplus \mathrm{Span}\{z\})\cap\Sigma_1$.
\end{proposition}
\begin{proof} Let $T\colon Y\to X_2$ be a bi-positive  linear map that satisfies $[Ty]=f(y)$ for all $y\in Y\cap\Sigma_1$ and choose
$z\in\Sigma_1\setminus Y$. Fix $\xi\in Y\cap\Sigma_1$ and consider $\ell^+_{z\xi}$. By Lemma \ref{two-dimensional base change} there exists a bi-positive
 linear map $S\colon\mathrm{Span}\{\xi, z\}\to X_2$ such that  $[Sx]=f(x)$ for all $x\in\ell^+_{z\xi}$.  By rescaling $S$, we may assume that $S\xi=T\xi$. 

Now let $\hat{Y}:= Y\oplus\mathrm{Span}\{z\}$ and define the linear map 
\[
\hat{T}\colon \hat{Y}\to\mathrm{ran\,}T+\mathrm{Span}\{Sz\}
\] 
by $y+\lambda z\mapsto Ty+\lambda Sz$. We wish to show that $\mathrm{ran\,}T+\mathrm{Span}\{Sz\}$ is in fact a direct sum, making $\hat{T}$ injective. 
Suppose $Ty=\lambda Sz$ for some $y\in Y$ and $\lambda>0$. As $T$ is bi-positive, this implies that $y\in Y\cap C_1^\circ$, so $[Ty]=[Sz]=f(z)$. But this yields
$f([y])=f(z)$, so $[y]=z$ and this is impossible. If $\lambda<0$, we can use a similar argument for $-y$ to arrive at a contradiction. Thus,
$\mathrm{ran\,}T+\mathrm{Span}\{Sz\}$ is indeed a direct sum. Since $\hat{T}\xi=T\xi=S\xi$ and $\hat{T}z=Sz$, it follows that $\hat{T}=S$ on $\ell^+_{z\xi}$
and so $[\hat{T}x]=f(x)$ for all $x\in\ell^+_{z\xi}$. 

Now suppose that $w\in\hat{Y}\cap\Sigma_1$ and  $w\notin Y\cup\ell^+_{z\xi}$. The
subspace $Y$ is a hyperplane in $\hat{Y}$ and therefore, it divides $\hat{Y}\cap\Sigma_1$ into two parts. Choose
distinct $\eta_1,\eta_2\in\ell^+_{z\xi}$ that lie on the other side of $w$ and let  $y_1,y_2\in Y\cap\Sigma_1$ be the intersection points of the line segments
$\ell^+_{w\eta_1}$ and $\ell^+_{w\eta_2}$, respectively. The situation is depicted in Figure \ref{collinear extension picture} below.
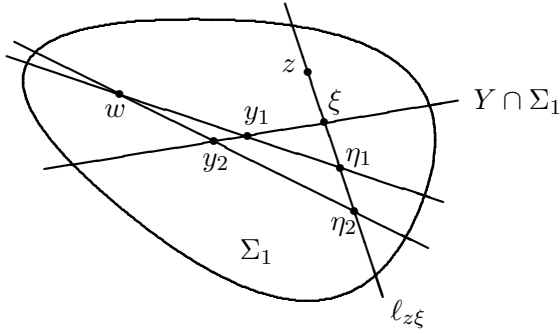
\begin{figure}[h]
\begin{center}
\thicklines
\setlength{\unitlength}{1mm}
\begin{picture}(60,40)
\closecurve(10,15, 20,40, 50,15)
\put(0,20){\line(6,1){55}}
\put(57,28){$Y\cap\Sigma_1$}
\put(32,42){\line(1,-3){13}}
\put(46,0){$\ell_{z\xi}$}
\put(-4,37){\line(2,-1){55}}
\put(-5,35){\line(3,-1){58}}
\put(10,30){\circle*{1.0}}\put(8,27){$w$}
\put(22.5,23.8){\circle*{1.0}}\put(21,20.5){$y_2$}
\put(27,24.5){\circle*{1.0}}\put(26.5,26.7){$y_1$}
\put(39.3,20.2){\circle*{1.0}}\put(40,21.2){$\eta_1$}
\put(41.2,14.4){\circle*{1.0}}\put(38,12){$\eta_2$}
\put(37.2,26.3){\circle*{1.0}}\put(37.7,28){$\xi$}
\put(35,33){\circle*{1.0}}\put(31.5,33){$z$}
\put(26,8){$\Sigma_1$}
\end{picture}
\caption{Points of intersection\label{collinear extension picture}}
\end{center}
\end{figure}
We see that $w$ is the unique point of intersection of the line segments $\ell^+_{\eta_1y_1}$ and $\ell^+_{\eta_2y_2}$. As $f$ is injective, it follows that
$f(w)$ is the unique point of intersection of the line segments $\ell^+_{f(\eta_1)f(y_1)}$ and $\ell^+_{f(\eta_2)f(y_2)}$. Also, as
$[\hat{T}\eta_i]=f(\eta_i)$ and $[\hat{T}y_i]=f(y_i)$, we have $\hat{T}\eta_i,\hat{T}y_i\in\mathrm{Span}\{f(\eta_i),f(y_i)\}$; hence
\[
\hat{T}w\in\mathrm{Span}\{f(\eta_1),f(y_1)\}\cap\mathrm{Span}\{f(\eta_2),f(y_2)\}=\mathrm{Span}\{f(w)\}.
\] 
So, $\hat{T}w\in C_2^\circ\cup- C_2^\circ$, since $\hat{T}$ is injective, and hence 
\[
\hat{T}(\hat{Y}\cap C_1^\circ)\subseteq C_2^\circ\cup-C_2^\circ.
\] 
The convexity of $\hat{Y}\cap C_1^\circ$ now implies that either $\hat{T}(\hat{Y}\cap C_1^\circ)\subseteq C_2^\circ$ or  $\hat{T}(\hat{Y}\cap C_1^\circ)\subseteq -C_2^\circ$. As $\hat{T}\xi=T\xi\in C_2^\circ$, it follows that
$\hat{T}(\hat{Y}\cap C_1^\circ)\subseteq C_2^\circ$. Moreover, this shows that $[\hat{T}w]$ is well defined. As $w$ was arbitrary, we conclude that
$[\hat{T}y]=f(y)$ for all $y\in\hat{Y}\cap\Sigma_1$.

Next, we will show that $\hat{T}$ is bi-positive. Suppose $x\in\partial C_1\cap \hat{Y}$ with $x\neq 0$. We claim that $\hat{T}(x)\in\partial C_2$. Since
$\varphi_1(x)>0$, we may assume without loss of generality that $x\in\partial\Sigma_1$. Define $x_n:=(1-\frac{1}{n})x+\frac{1}{n}\xi$ for $n\ge 1$. Then we have
$x_n\to x$, so that $f(x_n)\to f(x)$ by Lemma \ref{isometry in strictly convex cone} and
$$\hat{T}x_n=\hat{T}|_{\mathrm{Span}\{x,\xi\}}x_n\to\hat{T}|_{\mathrm{Span}\{x,\xi\}}x=\hat{T}x$$ as $\dim(\mathrm{Span}\{x,\xi\})<\infty$. It follows that
$\varphi_2(\hat{T}x_n)\to\varphi_2(\hat{T}x)$ and, since $\hat{T}x_n\in C_2$ for all $n\ge 1$, we must have $\hat{T}x\in C_2$, because $C_2$ is closed. Moreover,
the injectivity of $\hat{T}$ yields $\varphi_2(\hat{T}x)>0$. We also have that 
$$\hat{T}x_n=\varphi_2(\hat{T}x_n)f(x_n)\to\varphi_2(\hat{T}x)f(x)\in\partial C_2,$$ so we conclude that $\hat{T}x\in\partial C_2$ showing that $\hat{T}$ is
positive. Finally, if we pick $x\in \hat{Y}\setminus C_1$, then there exists a $0<t<1$ such that $tx+(1-t)\xi\in\partial C_1$ and so
$t\hat{T}x+(1-t)\hat{T}\xi\in\partial C_2$ by our previous findings. But this implies that $\hat{T}x\notin C_2$ making $\hat{T}$ bi-positive.
\end{proof}

Enough preparations have been made to prove the following result which implies Theorem 
\ref{thm:infinite_strictly_convex} by Theorem \ref{thm:cone_classification}.
\begin{theorem}\label{isometry on hilbert is collinear} An map $f\colon(\Sigma_1,d_H)\to(\Sigma_2,d_H)$  is an isometry that maps line segments to line segments if and only if
it is a projective linear homomorphism. 
\end{theorem}
\begin{proof}Consider the collection $\mathcal{C}$ of pairs $(Y,T_Y)$ where $Y\subseteq X_1$ is a linear subspace and $T_Y\colon Y\to X_2$ is a bi-positive
linear  map such that $[T_Yx]=f(x)$ for all $x\in Y\cap\Sigma_1$. Note that $\mathcal{C}\neq\emptyset$, since $C_1^\circ\neq\emptyset$. We can define a
partial order $\le$ on $\mathcal{C}$ by 
\[
(Y,T_Y)\le(Z,T_Z)\ \ \mbox{if}\ \ Y\subseteq Z\ \mbox{and}\ T_Zy=T_Yy\ \mbox{for all $y\in Y$}.
\] 
Let $(Y_i,T_{Y_i})_{i\in I}$ be a totally ordered subset in $\mathcal{C}$. Put $Y:=\bigcup_{i\in I}Y_i$ and define $T_Y\colon Y\to X_2$ through
$T_Yy_i:=T_{Y_i}y_i$. Clearly, we have that $Y$ is a linear subspace of $X_1$ and $T_Y$ is a well defined bi-positive  linear map. For $y\in
Y\cap\Sigma_1$ we have $y\in Y_i\cap\Sigma_1$ for some $i\in I$ and $[T_Yy]=[T_{Y_i}y]=f(y)$, so $(Y,T_Y)\in\mathcal{C}$ is an upper bound. By Zorn's lemma our
collection $\mathcal{C}$ contains a maximal element $(\Omega,T_\Omega)$. Suppose that $x\in \Sigma_1\setminus\Omega$. By Proposition \ref{extension of
collineation} we have a bi-positive  linear extension $$\hat{T}_\Omega\colon\Omega\oplus\mathrm{Span}\{x\}\to X_2$$ of $T_\Omega$ such that
$[\hat{T}_\Omega z]=f(z)$ for all $z\in(\Omega\oplus\mathrm{Span}\{x\})\cap\Sigma_1$. But now we have
$(\Omega\oplus\mathrm{Span}\{x\},\hat{T}_\Omega)\in\mathcal{C}$ and $(\Omega,T_\Omega)\leq(\Omega\oplus\mathrm{Span}\{x\},\hat{T}_\Omega)$, which contradicts
the maximality of $(\Omega,T_\Omega)$. We conclude that $\Omega\cap\Sigma_1=\Sigma_1$ and therefore $[T_\Omega x]=f(x)$ for all $x\in\Sigma_1$.

We claim that $\Omega=X_1$. Let $x\in X_1,\ u\in\Sigma_1$ and $\epsilon>0$ such that $B_\epsilon(u)\subseteq C_1^\circ$. Also, there are $y,z\in C_1$ such that
$x=y-z$, as $C_1-C_1=X_1$. Now, for $k\ge 1$ large enough, we have that $\frac{1}{k}y+u,\frac{1}{k}z+u\in B_\epsilon(u)$. Since $\Sigma_1\subseteq\Omega$, it
follows that $C_1^\circ\subseteq\Omega$; hence
$$\textstyle{\frac{1}{k}}x=\left(\textstyle{\frac{1}{k}}y+u\right)-\left(\textstyle{\frac{1}{k}}z+u\right)\in\Omega,$$ so $x\in\Omega$.

Obviously any projective linear homomorphism maps line segments to line segments, and hence we are done.
\end{proof}
Note that Lemma~\ref{lem:automatic_continuity} shows that the linear map $T$ in the previous theorem is continuous with respect to the order unit norm.
We also  note that Theorem \ref{isometry on hilbert is collinear} implies that two uniquely geodesic Hilbert geometries $\Omega_1$ and $\Omega_2$ are isometric if and only if there exists a projective linear isomorphism between them.
 
\begin{remark} An important variant of Hilbert's metric is Thompson's metric which  was introduced in \cite{Thom}. On the interior of a cone $C$ in an order unit space, {\em Thompson's metric} is given by 
\[
d_T(x,y):=\log \max\{M(x/y),M(y/x)\} \mbox{\quad for }x,y\in C^\circ.
\]
It was shown in \cite[Theorem 8.2]{LR} that if $C$ is a finite dimensional strictly convex cone with $\dim C\geq 3$, then for every isometry $f$ of $(C^\circ,d_T)$ there exists a bi-positive linear map $T\colon X\to X$ such that for each $x\in C^\circ$ we have that $f(x) =\lambda_x Tx$ for some $\lambda_x>0$. The proof of this result relies on \cite[Proposition 3]{dlH}. Using Theorem  \ref{isometry on hilbert is collinear} it is straightforward to extend \cite[Theorem 8.2]{LR} to infinite dimensional strictly convex cones. 
\end{remark} 

\section{Infinite dimensional simplices}
Let $K$ be a compact Hausdorff space and $C(K)$ denote the space of real-valued continuous functions on $K$. Consider the cone $C(K)_{+}$ consisting of
nonnegative functions with interior, 
\[
C(K)_{+}^\circ:=\{f\in C(K)\colon f(x)>0\mbox{ for all }x\in K\}.
\]

It is well known that the Hilbert geometry on a finite dimensional simplex  is isometric to a finite dimensional normed space. The same is true for $(P(C(K)_{+}^\circ),d_H)$, see \cite[Proposition 1.7]{NMem}. 
It will be useful to recall the basic argument.  Let  $\mathbf{1}$ be the constant one function on $K$ and denote the elements in the quotient space $  \overline{C(K)}  :=
C(K)/\mathbb{R}\mathbf{1}$ by $\overline{g}$. The map $\mathrm{Log}\colon P(C(K)_+^\circ)\to \overline{C(K)}$ given  by, 
\[
\mathrm{Log}(\overline{f}) := \overline{\log\circ f}\mbox{\quad for }\overline{f}\in P(C(K)^\circ_+),
\]
is an isometry of $(P(C(K))_+^\circ,d_H)$ onto $(\overline{C(K)},\|\cdot\|_{\mathrm{var}})$, where 
\[
\|\overline{g}\|_{\mathrm{var}} :=\sup_{x\in K} g(x) -\inf_{x\in K} g(x),
\]
is the {\em variation norm}. To see this note that  
\[
M(f/g)=\inf\{\beta>0:f(x)\le\beta g(x)\ \mbox{for all $x\in K$}\}=\sup_{x\in K}\frac{f(x)}{g(x)}\] 
for $f,g\in P(C(K)_+^\circ)$, so that 
\begin{align*}d_H(f,g)&=\sup_{x\in K}\frac{\log f(x)}{\log g(x)}+\sup_{x\in K}\frac{\log g(x)}{\log f(x)}\\&=\sup_{x\in K}\left(\log f(x)-\log g(x)\right)-\inf_{x\in K}\left(\log f(x)-\log g(x)\right)\\&=\|\mathrm{Log}(\overline{f})-\mathrm{Log}(\overline{g})\|_{\mathrm{var}}.
\end{align*}

Given a function $f\in C(K)$ the supremum norm of its translations $f-\lambda\mathbf{1}$ for $\lambda\in\mathbb{R}$ is minimised precisely when translating $f$ by the average of both extreme values $\frac{1}{2}(\sup_{x\in K}f(x)+\inf_{x\in K}f(x))$. This means that the quotient norm on $C(K)/\mathbb{R}\mathbf{1}$ with respect to the supremum norm is exactly half of $\|\cdot\|_{\mathrm{var}}$. This assertion is made precise in the following lemma.
\begin{lemma}\label{sup norm on quotient} Let $K$ be a compact Hausdorff space. If $\|\cdot\|_q$ is the quotient norm on $\overline{C(K)}$
with respect to $2\|\cdot\|_\infty$, then $\|\cdot\|_q$ coincides with $\|\cdot\|_{\mathrm{var}}$ on $\overline{C(K)}$.
\end{lemma}
\begin{proof}
 Let $f \in C(K)$ and write $a\vee b=\max\{a,b\}$ for $a,b\in\R$. Using the elementary fact that $\inf_{\lambda \in \R} (a-\lambda) \vee (b + \lambda) = (a+b)/2$, we see that
 \begin{align*}
  \norm{\overline{f}}_q &= 2 \inf_{\lambda \in \R} \norm{f - \lambda\mathbf{1}}_\infty = 2 \inf_{\lambda \in \R} \sup_{s \in K} |f(s) - \lambda| \\
&= 2 \inf_{\lambda \in \R} \left[ \left( \sup_{s \in K} f(s) - \lambda \right) \vee \left( \sup_{s \in K} -f(s) + \lambda \right) \right] \\
&= \sup_{s \in K} f(s) + \sup_{s \in K} -f(s) = \sup_{s \in K} f(s) - \inf_{s \in K} f(s) = \norm{\overline{f}}_{\mathrm{var}}.
 \end{align*}
\end{proof}

Our findings so far have shown that in order to describe the surjective isometries $$h\colon(P(C(K_1)^\circ_+),d_H)\to(P(C(K_2)^\circ_+),d_H),$$ it suffices to
understand the surjective isometries $$T\colon(\overline{C(K_1)},\|\cdot\|_q)\to(\overline{C(K_2)},\|\cdot\|_q).$$ By the Mazur-Ulam theorem these
isometries $T$ must be affine. We can compose $T$ with an appropriate translation to make it linear. Thus, our goal will be to classify all 
isometric isomorphisms (surjective linear isometries) $T\colon(\overline{C(K_1)},\|\cdot\|_q)\to(\overline{C(K_2)},\|\cdot\|_q)$.

We will follow the lines of the proof for the Banach-Stone theorem \cite[Theorem VI.2.1]{Con}. The Banach-Stone theorem characterises the isometric isomorphisms
between $(C(K),\|\cdot\|_\infty)$ spaces. A common way to prove the Banach-Stone theorem is by looking at the adjoint operator, which is an isometry on the dual
of $(C(K),\|\cdot\|_\infty)$, and to exploit the extreme points of the unit ball there. We shall take a similar approach. 

The dual of $(C(K),\|\cdot\|_\infty)$ is $(M(K),\|\cdot\|_{TV})$, where $M(K)$ is the space of all regular signed Borel measures on $K$ and
$\|\mu\|_{TV}:=|\mu|(K)$ is the {\em total variation norm}.  Let us recall some basic facts about $(M(K),\|\cdot\|_{TV})$, which can be found in \cite[Appendix
C]{Con}. 

Every $\mu\in M(K)$ has a Hahn-Jordan decomposition $\mu=\mu^+-\mu^-$ where $\mu^+$ and $\mu^-$ are positive measures in $M(K)$, and \begin{align}\label{total
variation}\|\mu\|_{TV}=\|\mu^+\|_{TV}+\|\mu^-\|_{TV}.\end{align} Also, if $\mu,\nu\in M(K)$ are positive, we have $\|\mu+\nu\|_{TV}=\|\mu\|_{TV}+\|\nu\|_{TV}$.
The space $M(K)$ is a lattice: every $\mu, \nu \in M(K)$ have a supremum (least upper bound) $\mu \vee \nu$. The set $P(K)$ denotes the set of
probability measures on $K$, and its extreme points are the set of Dirac measures $\{\delta_s \colon s \in K\}$. The map $s \mapsto \delta_s$ is a
homeomorphism from $K$ onto $\{\delta_s \colon s \in K\}$ equipped with the weak*-topology.

The dual space of $(C(K)/\mathbb{R}\mathbf{1},\|\cdot\|_q)$ is
$\mathbb{R}\mathbf{1}^\perp\subseteq(M(K),\frac{1}{2}\|\cdot\|_{TV})$, where $\mathbb{R}\mathbf{1}^\perp :=\{\mu \in M(K)\colon \mu(K)=0\}$. It follows that  
$$
\mathbb{R}\mathbf{1}^\perp=\{\mu\in M(K):\|\mu^+\|_{TV}=\|\mu^-\|_{TV}\}.
$$

Now, if $$T\colon (\overline{C(K_1)},\|\cdot\|_q)\to(\overline{C(K_2)},\|\cdot\|_q)$$ is an isometric isomorphism, then the
corresponding adjoint operator
$$T^*\colon(\mathbb{R}\mathbf{1}_2^\perp,\textstyle{\frac{1}{2}}\|\cdot\|_{TV})\to(\mathbb{R}\mathbf{1}_1^\perp,\textstyle{\frac{1}{2}}\|\cdot\|_{TV})$$ is a
isometric isomorphism as well. Moreover, $T^*$ is a weak*-homeomorphism from the unit ball $B_2\subseteq\mathbb{R}\mathbf{1}_2^\perp$ onto the unit ball
$B_1\subseteq\mathbb{R}\mathbf{1}_1^\perp$ that maps the set of extreme points of $B_2$, denoted $\mathrm{ext}(B_2)$, bijectively onto the set  of the extreme points of $B_1$, denoted $\mathrm{ext}(B_1)$. The following lemma tells us  that the extreme points  are exactly the differences of Dirac measures.

\begin{proposition}\label{extreme points}Let $K$ be a compact Hausdorff space. The set of extreme points, $\mathrm{ext}(B)$, of the unit sphere $B$ in
$\mathbb{R}\mathbf{1}^\perp\subseteq(M(K),\frac{1}{2}\|\cdot\|_{TV})$ satisfies  $$\mathrm{ext}(B)=\left\{\delta_s-\delta_t: s,t\in K\mbox{ and }s\neq
t\right\}.$$
\end{proposition}
\begin{proof}
Let $\delta_s-\delta_t\in B$ with $s\neq t$ and suppose $\mu,\nu\in B$ are such that 
\[ \delta_s-\delta_t=\frac{1}{2} \left(\mu+\nu \right) = \frac{1}{2}\left(\mu^+ + \nu^+ \right) - \frac{1}{2} \left( \mu^- + \nu^- \right). \]
Then $\delta_s = (\delta_s-\delta_t) \vee 0 \leq \textstyle{\frac{1}{2}}(\mu^++\nu^+)$, and so for $\eta := \frac{1}{2}(\mu^+ + \nu^+) - \delta_s \geq 0$ it
follows from \eqref{total variation} that 
\[ 1 + \|\eta\|_{TV} = \|\delta_s+\eta\|_{TV}=\textstyle{\frac{1}{2}}\|\mu^++\nu^+
\|_{TV}=\textstyle{\frac{1}{2}}\|\mu^+\|_{TV}+\textstyle{\frac{1}{2}}\|\nu^+\|_{TV}
=1, \]
 so $\eta=0$, which yields $\delta_s=\frac{1}{2}(\mu^++\nu^+)$. The fact that $\delta_s$ is an extreme point in $B_{M(K)}$ (see \cite[Theorem
V.8.4]{Con}) implies that $\mu^+=\nu^+$. Similarly, we have that $\delta_t\le\frac{1}{2}(\mu^-+\nu^-)$ which implies $\mu^-=\nu^-$. This shows that
$\delta_s-\delta_t\in\mathrm{ext}(B)$.

Conversely, let $\mu = \mu^+ - \mu^- \in \mathrm{ext}(B)$, and suppose $\mu^+ = \frac{1}{2}(\nu_1 + \nu_2)$ for $\nu_1, \nu_2 \in P(K)$. Then for $i = 1,2$,
\[ \frac{1}{2} \norm{\nu_i - \mu^-}_{TV} \leq \frac{1}{2}\left( \norm{\nu_i}_{TV} + \norm{\mu^-}_{TV}\right) = 1, \]
so $\nu_i - \mu^- \in B$. Since $\mu \in \mathrm{ext}(B)$ and $\mu = \frac{1}{2}(\nu_1 - \mu^-) + \frac{1}{2}(\nu_2 - \mu^-)$, we obtain $\nu_1 = \nu_2$. Hence
$\mu^+ \in \mathrm{ext}(P(K))$, which implies that $\mu^+ = \delta_s$ for some $s \in K$ by \cite[Theorem V.8.4]{Con}. A similar argument yields $\mu^- = \delta_t$ for some $t \in K$.
\end{proof}

For $s\in K$ define the sets $E_s:=\{\delta_s-\delta_t: t\in K,\ s\neq t\}$. Clearly the distance between distinct elements of $E_s$ is 1, and it turns
out that the sets $\pm E_s$ are the maximal equilateral subsets in $\mathrm{ext}(B)$ of mutual distance 1.

\begin{lemma}\label{maximal sets in ext(B)} Let $K$ be a compact Hausdorff space and $\emptyset\neq A\subseteq\mathrm{ext}(B)$ be such that $\frac{1}{2}\|\mu-\nu\|_{TV}=1$ for all $\mu,\nu\in A$ with $\mu\neq\nu$, then there is an element $s\in K$ such that $A\subseteq E_s$ or $A\subseteq -E_s$. 
\end{lemma}
\begin{proof}
If such an $s$ does not exist, then there exist elements $\delta_s - \delta_t, \delta_p - \delta_q \in A$ with $s \not= p$ and $t \not= q$.
But clearly 
\[
\frac{1}{2}\norm{(\delta_s - \delta_t) - (\delta_p - \delta_q)}_{TV} = 2.
\] 
\end{proof} 

 Now let $T \colon \overline{C(K_1)} \to \overline{C(K_2)}$ be an isometric isomorphism. Then the isometric isomorphism $T^*$ preserves the maximal equilateral subsets of the extreme points of mutual distance 1. Hence $T^*(E_s) = \pm E_{\theta(s)}$. Note that  if $s \not= t$, then $E_s \cap E_t = \emptyset$. But $E_s \cap -E_t \not= \emptyset$, as $\delta_s - \delta_t \in E_s\cap -E_t$. As $T^*$ maps disjoint sets to disjoint sets, either $T^*(E_s) = E_{\theta(s)}$ for all $s \in K_2$, or, $T^*(E_s) = -E_{\theta(s)}$ for all $s \in
K_2$. Thus, there exists  $\eps \in \{-1,1\}$ such that  $T^*(E_s) = \eps E_{\theta(s)}$ for all $s\in K_2$, and $\theta$ is a bijection from $K_2$ to $K_1$.

\begin{lemma}\label{homeomorphism of K}
The above constructed bijection $\theta \colon K_2 \to K_1$ is a homeomorphism.
\end{lemma}
\begin{proof}
Let $(s_\alpha)_\alpha$ be a net in $K_2$ converging to $s \in K_2$. Then $\delta_{s_\alpha}$ converges weak* to $\delta_s$, and so
\[ \delta_{\theta(s_\alpha)} - \delta_{\theta(s)} = \eps T^*(\delta_{s_\alpha} - \delta_s) \stackrel{\text{weak*}}{\longrightarrow} \eps T^* 0 = 0. \]
Hence $\delta_{\theta(s_\alpha)} \stackrel{\text{weak*}}{\longrightarrow}  \delta_{\theta(s)}$, or equivalently, $\theta(s_\alpha) \to \theta(s)$. So $\theta$ is a continuous bijection from a
compact space into a Hausdorff space and so it is a homeomorphism.
\end{proof}

We are now able to prove our main result on isometric isomorphisms $T \colon \overline{C(K_1)} \to \overline{C(K_2)}$.
\begin{theorem}\label{thm:isometric_isomorphism}
 If $K_1$ and $K_2$ are compact Hausdorff spaces, then  a map $T \colon \overline{C(K_1)} \to \overline{C(K_2)}$ is an isometric isomorphism if and only
if there exist an $\eps \in \{-1,1\}$ and a homeomorphism $\theta \colon K_2 \to K_1$ such that $T \overline{f} = \overline{\eps (f \circ \theta)}$.
\end{theorem}
\begin{proof}
 Suppose $T$ is an isometric isomorphism, and let $\theta$ and $\eps$ be such that $T^* E_s = \eps E_{\theta(s)}$. Although point-evaluation for elements
$\overline{g} \in \overline{C(K_2)}$ is not well defined, the values $\overline{g}(s) - \overline{g}(t)$ for $s,t \in K_2$ are well defined, and the
computation
\begin{align*}
 T \overline{f}(s) - T \overline{f}(t) &= (\delta_s - \delta_t)(T \overline{f}) \\
 &= [T^*(\delta_s - \delta_t)](\overline{f}) \\
 &= [\eps (\delta_{\theta(s)} - \delta_{\theta(t)})] (\overline{f}) \\
 &= \eps(\overline{f}(\theta(s)) - \overline{f}(\theta(t)))
\end{align*}
shows that $T$ is induced from the map $f \mapsto \eps (f \circ \theta)$ from $C(K_1)$ into $C(K_2)$.

Conversely, if $\eps \in \{-1,1\}$ and $\theta \colon K_2 \to K_1$ is a homeomorphism, then $f \mapsto \eps (f \circ \theta)$ is an isometric isomorphism
between $C(K_1)$ and $C(K_2)$ that maps $\R \mathbf{1}_1$ onto $\R \mathbf{1}_2$, and hence it induces an isometric isomorphism between the respective quotient
spaces $\overline{C(K_1)}$ and $\overline{C(K_2)}$.
\end{proof}

Our next goal is to describe the group of surjective isometries from $\overline{C(K)}$ to itself. For any real normed space $X$, the Mazur-Ulam theorem shows that any surjective isometry from $X$ to $X$ is the composition of an isometric isomorphism and a translation. Clearly the subgroup of isometric isomorphisms and the subgroup of translations have trivial intersection, and it is easily verified that a translation by $x \in X$ conjugated by an isometric isomorphism $T$ of $X$ yields a translation by $Tx$. This shows that the group of surjective isometries from $X$ to $X$ is a semidirect product of these two subgroups. In the case of $X = \overline{C(K)}$, we know the group of isometric isomorphisms by Theorem~\ref{thm:isometric_isomorphism} which yields the following description.
\begin{proposition}\label{prop:isometry_group}
 If $K$ be a compact Hausdorff space, then the group of surjective isometries of $\overline{C(K)}$ is isomorphic to $\overline{C(K)} \rtimes (C_2 \times \mathrm{Homeo}(K))$ if and only if $|K| \geq 3$, where $C_2$ is the cyclic group of order 2. 
\end{proposition}
Note that in the above proposition, $\theta \in \mathrm{Homeo}(K)$ acts on $\overline{f} \in \overline{C(K)}$ as $\overline{f \circ \theta^{-1}}$ (not as
$\overline{f \circ \theta}$).
\begin{proof}
 If $K$ has only 2 elements, then multiplication by $-1$ coincides with the non-trivial homeomorphism, and if $|K| \leq 1$, then multiplication by $-1$ is the
identity. In all other cases there is no overlap, and the result follows from the above discussion.
\end{proof}

Translating Theorem~\ref{thm:isometric_isomorphism} and Proposition~\ref{prop:isometry_group} back to $(P(C(K_i)^\circ_+), d_H)$ through the pointwise
exponential (the inverse of the pointwise logarithm) yields the following. 

\begin{theorem}\label{hilbert isometry on function space} Let $K_1,K_2$ be compact Hausdorff spaces. A map $$h\colon(P(C(K_1)^\circ_+),d_H)\to
(P(C(K_2)^\circ_+),d_H)$$ is a surjective isometry if and only if there exist  $g\in C(K_2)^\circ_+$, $\epsilon \in\{ -1, 1\}$, and a homeomorphism $\theta\colon K_2\to K_1$ such
that 
\[
h(\overline{f})=\overline{g\cdot(f\circ\theta)^\epsilon}\mbox{\quad for all }\overline{f}\in P(C(K_2)^\circ_+).
\] 
If $K_1=K_2=K $  and $|K| \geq 3$, then the isometry group  is given by 
\[
\mathrm{Isom}(P(C(K)_+^\circ), d_H) \cong \overline{C(K)} \rtimes (C_2 \times \mathrm{Homeo}(K)),
\]
where $C_2$ is the cyclic group of order $2$. 
\end{theorem} 
If $\mu_i$ ($i \in \{1,2\}$) is a strictly positive measure on $K_i$, and we identify $P(C(K_i)^\circ_+$ with $\Delta(K_i, \mu_i)$, we obtain
Theorem~\ref{thm:1.2}.
\begin{remark}
 In the above theorem, everything is projectively linear except for the inversion. It follows that if $K$ is a compact Hausdorff space with at least 3
elements, the index of the collineation group in the isometry group $\mathrm{Isom}(P(C(K)^\circ_+),d_H)$ equals 2.
\end{remark}

Theorem \ref{hilbert isometry on function space} has the following interesting consequence. 
\begin{corollary}
If $K_1, K_2$ are compact Hausdorff spaces,  then $(P(C(K_1)_+^\circ),d_H)$ and $(P(C(K_2)_+^\circ),d_H)$ are isometric if and only if $K_1$ and $K_2$ are homeomorphic.
\end{corollary}

It would be interesting to study non-commutative versions of Theorem \ref{hilbert isometry on function space}. In particular, one could  look at Hilbert's metric isometries on 
the interior of the cone $A_+$ of positive self-adjoint elements in a unital $C^*$-algebra.  In view of the characterisation of  Thompson's metric isometries on $A_+^\circ$ by Hatori and Moln\'ar \cite{HM}, it seems plausible that each Hilbert's metric isometry $h\colon P(A_+^\circ)\to P(A_+^\circ)$ is of the form $h(\overline{a}) =\overline{h(e)^{1/2}J(a^{\epsilon})h(e)^{1/2}}$, 
where $\epsilon \in \{-1,1\}$ is fixed, $e$ is the unit in $A$, and $J$ is a Jordan* isomorphism. 

More generally is seems worthwhile to investigate if Walsh's results in \cite{W2} can be extended to infinite dimensions. In \cite{W2} Walsh showed that every isometry of a finite dimensional Hilbert geometry is a projective linear automorphism, except when the domain comes from a non-Lorentzian symmetric cone. It is well known that the symmetric cones in finite dimensional vector spaces are precisely the interiors of the cones of squares of Euclidean Jordan algebras by the fundamental work of Koecher \cite{Koe} and Vinberg \cite{Vin}. Thus, Walsh's result provides a link between the existence of an isometry of $(\Omega,d_H)$ that is not a projective linear automorphism and a Jordan algebra structure on the vector space above $\Omega$. It might well be true that in a general order unit space $(X,C,u)$ we have that the existence of  a Hilbert's metric  isometry on $P(C^\circ)$ that is not a projective linear automorphism  implies  that $X$ has a Jordan algebra structure and $C$ is the cone of squares.

\end{document}